\documentclass[11pt,reqno,tbtags,draft]{amsart}
\usepackage{amssymb}
\usepackage{url}
\usepackage[square,numbers]{natbib}
\bibpunct[, ]{[}{]}{;}{n}{,}{,}

\title
{Maximal clades in random binary search trees}

\date{27 August, 2014}

\author{Svante Janson}
\thanks{Partly supported by the Knut and Alice Wallenberg Foundation}
\address{Department of Mathematics, Uppsala University, PO Box 480,
SE-751~06 Uppsala, Sweden}
\email{svante.janson@math.uu.se}
\newcommand\urladdrx[1]{{\urladdr{\def~{{\tiny$\sim$}}#1}}}
\urladdrx{http://www2.math.uu.se/~svante/}
\thanks{This research was started during the 
25th International Conference on Probabilistic, Combinatorial and Asymptotic
Methods for the Analysis of Algorithms, AofA 2014, 
in Paris, June 2014.
I thank Michael Fuchs for interesting discussions.
}

\subjclass[2010]{60C05; 05C05, 60F05}

\numberwithin{equation}{section}

\renewcommand\le{\leqslant}
\renewcommand\ge{\geqslant}

\allowdisplaybreaks

\newtheorem{theorem}{Theorem}[section]
\newtheorem{lemma}[theorem]{Lemma}

\theoremstyle{definition}

\newtheorem{remark}[theorem]{Remark}

\theoremstyle{remark}

\newenvironment{romenumerate}[1][0pt]{
\addtolength{\leftmargini}{#1}\begin{enumerate}
 }{\end{enumerate}}

\newcounter{oldenumi}
{\setcounter{oldenumi}{\value{enumi}}
\begin{romenumerate} \setcounter{enumi}{\value{oldenumi}}}
{\end{romenumerate}}

\newcounter{thmenumerate}
\newenvironment{thmenumerate}
{\setcounter{thmenumerate}{0}%
 \def\item{\par
 \refstepcounter{thmenumerate}\textup{(\roman{thmenumerate})\enspace}}
}
{}

\newcounter{romxenumerate}   

\newcounter{xenumerate}   

\newcommand\step[1]{\par{#1.}}
\newcounter{steps}
\newcommand\stepx{\smallskip\noindent\refstepcounter{steps}%
 \emph{Step \arabic{steps}. }}

\newcommand{\refT}[1]{Theorem~\ref{#1}}

\newcommand{\refL}[1]{Lemma~\ref{#1}}
\newcommand{\refR}[1]{Remark~\ref{#1}}
\newcommand{\refS}[1]{Section~\ref{#1}}

\begingroup
  \divide\count255 by 60
  \multiply\count255 by -60
  \advance\count255 by \time
  \ifnum \count255 < 10 \xdef\klockan{\the\count1.0\the\count255}
  \else\xdef\klockan{\the\count1.\the\count255}\fi
\endgroup

\newcommand{\sumji}{\sum_{j=1}^\infty}

\newcommand{\sumni}{\sum_{n=1}^\infty}
\newcommand{\sumki}{\sum_{k=1}^\infty}
\newcommand{\sumkix}[1]{\sum_{k=1}^{#1}}

\newcommand{\sumkn}{\sum_{k=1}^n}
\newcommand{\sumkN}{\sum_{k=1}^N}
\newcommand{\sumkni}{\sum_{k=1}^{n-1}}

\newcommand{\prodim}{\prod_{i=1}^m}

\newcommand\set[1]{\ensuremath{\{#1\}}}

\newcommand\xpar[1]{(#1)}
\newcommand\bigpar[1]{\bigl(#1\bigr)}
\newcommand\Bigpar[1]{\Bigl(#1\Bigr)}
\newcommand\biggpar[1]{\biggl(#1\biggr)}
\newcommand\lrpar[1]{\left(#1\right)}

\newcommand\xcpar[1]{\{#1\}}

\newcommand\abs[1]{|#1|}
\newcommand\bigabs[1]{\bigl|#1\bigr|}
\newcommand\Bigabs[1]{\Bigl|#1\Bigr|}

\def\rompar(#1){\textup(#1\textup)}    
\newcommand\xfrac[2]{#1/#2}

\newcommand\parfrac[2]{\lrpar{\frac{#1}{#2}}}

\newcommand\Bigparfrac[2]{\Bigpar{\frac{#1}{#2}}}

\def\xexp(#1){e^{#1}}

\newcommand\floor[1]{\lfloor#1\rfloor}

\newcommand\ntoo{\ensuremath{{n\to\infty}}}

\newcommand\ktoo{\ensuremath{{k\to\infty}}}

\newcommand\norm[1]{\|#1\|}
\newcommand\normp[1]{\|#1\|_p}
\newcommand\bignorm[1]{\bigl\|#1\bigr\|}

\newcommand\punkt{.\spacefactor=1000}    
    
\newcommand\ie{i.e\punkt}
\newcommand\eg{e.g\punkt}

\newcommand\cf{cf\punkt}

\newcommand{\tend}{\longrightarrow}
\newcommand\dto{\overset{\mathrm{d}}{\tend}}
\newcommand\pto{\overset{\mathrm{p}}{\tend}}

\newcommand\eqd{\overset{\mathrm{d}}{=}}

\newcommand\op{o_{\mathrm p}}

\newcounter{CC}
\newcommand{\CC}{\stepcounter{CC}\CCx} 
\newcommand{\CCx}{C_{\arabic{CC}}}     
\newcommand{\CCname}[1]{\CC\xdef#1{\CCx}} 
\newcounter{cc}

\newcommand\E{\operatorname{\mathbb E{}}}
\newcommand\PP{\operatorname{\mathbb P{}}}
\newcommand\Var{\operatorname{Var}}

\newcommand\Exp{\operatorname{Exp}}

\newcommand\rise[1]{^{\overline{#1}}}

\newcommand\ga{\alpha}

\newcommand\gd{\delta}
\newcommand\gD{\Delta}
\newcommand\gf{\varphi}
\newcommand\gam{\gamma}

\newcommand\gl{\lambda}

\newcommand\gs{\sigma}
\newcommand\gss{\sigma^2}
\newcommand\eps{\varepsilon}

\renewcommand\phi{\xxx}  

\newcommand\cF{\mathcal F}

\newcommand\cT{{\mathcal T}}

\newcommand\ett[1]{\boldsymbol1\xcpar{#1}}

\newcommand\qw{^{-1}}
\newcommand\qww{^{-2}}
\newcommand\qq{^{1/2}}
\newcommand\qqc{^{3/2}}

\newcommand\intoi{\int_0^1}

\newcommand\oi{[0,1]}

\newcommand\dd{\,\mathrm{d}}

\newcommand\rhs{right-hand side}

\newcommand\sL{{\mathsf L}}
\newcommand\sR{{\mathsf R}}
\newcommand\tl{T_\sL}
\newcommand\tr{T_\sR}
\newcommand\TT{\bar \cT}
\newcommand\ttn{\TT_n}
\newcommand\ctn{\cT_n}
\newcommand\ctk{\cT_k}
\newcommand\ctx[1]{\cT_{#1}}
\newcommand\ctnx[1]{\ctx{n,#1}}
\newcommand\ctnv{\ctnx v}
\newcommand\ctnl{\cT_{n,\sL}}
\newcommand\ctnr{\cT_{n,\sR}}
\newcommand\ccT{\tilde\cT}
\newcommand\ctt{\ccT_t}
\newcommand\cttau{\ccT_\tau}
\newcommand\vvx[1]{v_1\dotsm v_{#1}}
\newcommand\vvk{\vvx{k}}
\newcommand\Too{T_\infty}
\newcommand\Voo{V_\infty}
\newcommand\Cl{C_\sL}
\newcommand\Cr{C_\sR}
\newcommand\ctgl{\cT^{(\gl)}}
\newcommand\Fii{{}_1F_1}
\newcommand\gamm{\gam^2}
\newcommand\bst{binary search tree}
\newcommand\rbst{random \bst}
\newcommand\dV{\partial V}
\newcommand\gff{\gf'}
\newcommand\pxx[1]{\frac{2}{(#1+1)(#1+2)}}
\newcommand\pkk{\pxx{k}}
\newcommand\sumvtn{\sum_{v\in \ctn}}
\newcommand\sumvv{\sum_{v\in V_m'}}
\newcommand\sumvdv{\sum_{v\in\dV_m'}}
\newcommand\qp{^{1/p}}
\newcommand\qpp{^{1-1/p}}
\newcommand\FF{F^{(p)}}

\newcommand{\Holder}{H\"older}

\begin{document}





\begin{abstract} 
We study maximal clades in random phylogenetic trees with the Yule--Harding
model or, equivalently, in binary search trees. We use probabilistic methods
to reprove and extend earlier results on moment asymptotics and asymptotic
normality. In particular, we give an explanation of the curious phenomenon
observed by Drmota, Fuchs and Lee (2014) that asymptotic normality holds, but
one should normalize using half the variance.
\end{abstract}

\maketitle

\section{Introduction}\label{S1}

Recall that there are two types of binary trees; we fix the notation as follows.
A \emph{full binary tree} is an rooted tree where each node has
either 0 or 2 children; in the latter case the two children are designated as
\emph{left child} and \emph{right child}.
A \emph{binary tree} is a  rooted tree where each node has
0, 1 or 2 children; moreover, each child is designated as either 
\emph{left  child} or \emph{right child}, and each node has at most one
child of each type. 
(Both versions can be regarded as ordered trees,  
with the left child before the right when there are two children.)
It is convenient to regard also the empty tree $\emptyset$
as a binary tree (but not as
a full binary tree). 
In a full binary tree, the leaves (nodes with no children) 
are called \emph{external nodes}; the other nodes (having 2 children) are
\emph{internal nodes}.
There is a simple, well-known bijection between full binary trees and binary
trees: Given a full binary tree, its internal nodes form a binary tree;
this is a bijection, with inverse given by adding, to
any given binary tree, external nodes as children at all free places.

Note that a full binary tree with $n$ internal nodes has $n+1$ external
nodes, and thus $2n+1$ nodes in total. In particular, the bijection just
described yields a bijection between the full binary trees with $2n+1$ nodes
and the binary trees with $n$ nodes.

If $T$ is a binary, or full binary, tree, we let $T_\sL$ and $T_\sR$
be the subtrees rooted at the left and right child of the root, with
$T_\sL=\emptyset$  [$T_\sR=\emptyset$] if the root has no left [right] child.

A \emph{phylogenetic tree} is the same as a  full binary tree.
In this context,  the \emph{clade} of an external node $v$ is defined 
to be the set of external nodes that are descendants of the parent of $v$.
(This is called a \emph{minimal clade} by \citet{BlumF} and \citet{ChangF10}.)
Note that two clades are either nested or disjoint; furthermore, each
external node belongs to some clade (for example its own).
Hence, the set of maximal
clades forms a partition of the set of external nodes.
We let $F(T)$ denote the number of maximal clades of a phylogenetic
tree $T$.  (Except that for technical reasons, see \refS{Sbin}, we define
$F(T)=0$ for a phylogenetic tree $T$ with only one external node. Obviously,
this does not affect asymptotics.)
The maximal clades, and the number of them, were introduced by \citet{DBF07},
together with a biological motivation, and further studied by \citet{DFL}.


The phylogenetic trees that we consider are random; more precisely, we 
consider the Yule--Harding model of a random phylogenetic tree $\TT_{n}$
with a given number $n$ internal, and thus $n+1$ external, nodes. 
These can be defined recursively, with $\TT_0$ the
unique phylogenetic tree with 1 node (the root), and $\TT_{n+1}$ obtained
from $\TT_n$ ($n\ge0$)
by choosing an external node uniformly at random and converting
it to an internal node with two external children.
(Alternatively, we obtain the same random model by constructing the tree
bottom-up by Kingman's coalescent \cite{Kingman}, 
see further \citet{Aldous-cladograms}, \citet{BlumF} and \citet{ChangF10}.) 
Recall that, for any $n\ge1$, the number of internal nodes in the 
left subtree $\TT_{n,\sL}$ (or the right subtree $\TT_{n,\sR}$) is uniformly
distributed on $\set{0,\dots,n-1}$, 
and that conditioned on this number being $m$, $\TT_{n,\sL}$ has the same
distribution as $\TT_m$;
see also \refR{Rsplit}.

Under the bijection above, the Yule--Harding random tree $\TT_n$ corresponds
to the random \emph{binary search tree} $\cT_n$ with $n$ nodes, see \eg{}
\citet{SJ180} and \citet{Drmota}.

The random variable that we study is thus $X_n:=F(\ttn)$,
the number of maximal clades in the Yule--Harding model.
It was proved by \citet{DF10}
that the mean number of maximal clades 
$\E X_n\sim\ga n$,
where
\begin{equation}\label{ga}
  \ga = \frac{1-e\qww}4.
\end{equation}
This was reproved by \citet{DFL},
in a sharper form:
\begin{theorem}[\cite{DF10,DFL}]
  \label{Tmean}
  \begin{equation}\label{tmean}
\E X_n=\E F(\ctn)=\ga n+O(1),  
\end{equation}
where $\ga$ is given by \eqref{ga}.
\end{theorem}

Moreover, \citet{DFL} found also corresponding
results for the variance and higher central moments:
\begin{theorem}[\cite{DFL}]
\label{Tmom}  
As \ntoo,
\begin{align}\label{Xvar}
  \E\xpar{X_n-\E X_n}^2 
&\sim 	4\ga^2 n\log n, &&  
\intertext{and for any fixed integer $k\ge3$,}
  \E\xpar{X_n-\E X_n}^k &\sim (-1)^k\frac{2k}{k-2}\ga^k n^{k-1}.
\label{Xmom}
\end{align}
\end{theorem}

As a consequence of \eqref{Xvar}--\eqref{Xmom}, the limit distribution of
$F(\ttn)$ (after centering and normalization) cannot be found by the method
of moments.  
Nevertheless,
\cite{DFL} further proved asymptotic normality,
where, unusually,  the normalizing uses (the square root of) \emph{half} the
variance:
\begin{theorem}[\cite{DFL}]  \label{TCLT}
As \ntoo,
\begin{equation}\label{tclt}
  \frac{X_n-\E X_n}{\sqrt{2\ga^2 n\log n}} \dto N(0,1).
\end{equation}
\end{theorem}

Here and below, $\dto$ denotes convergence in distribution;
similarly, $\pto$ will denotes convergence in probability. 
Unspecified limits (including implicit ones such as $\sim$ and $o(1)$) will
be as \ntoo.
Furthermore, $Y_p=\op(a_n)$, for random variables $Y_n$ and
positive numbers $a_n$, means $Y_n/a_n\pto0$.
We let $C,C_1,C_2,\dots$ denote some unspecified positive constants.

The purpose of the present paper is to  
use probabilistic methods to reprove these theorems, together with some
further results;
we hope that this can give additional insight, 
and it might perhaps also suggest future
generalizations to other types of random trees.

In particular, we can explain the appearance of half the variance in
\refT{TCLT} as follows: 

Fix a sequence of numbers $N=N(n)$, and say that a clade is \emph{small} if
it has at most $N+1$ elements, and \emph{large} otherwise.
(We use $N+1$ in the definition only for later notational convenience; 
the subtree corresponding to a small clade has at most $N$ internat nodes.)
Let $X^N_n$ be the number of maximal small clades, \ie, the small clades
that are not contained in any other small clade. It turns out that a
suitable choice of $N$ is about $\sqrt n$; we give two versions in the next
theorem. 
\begin{theorem}\label{Tsmall}
  \begin{thmenumerate}
  \item \label{tsmall12}
Let $N:=\sqrt{n}$. 
Then $\Var(X^N_n)\sim 2\ga^2 n\log n$	and
\begin{equation}\label{tsmall}
  \frac{X^N_n-\E X^N_n}{\sqrt{\Var X_n^N}} \dto N(0,1).
\end{equation}
Furthermore,
$X_n-X_n^N=\op\bigpar{\sqrt{\Var X_n^N}}$ and
$\E X_n-\E X_n^N=o\bigpar{\sqrt{\Var X_n^N}}$, so we may replace $X_n^N$
by $X_n$ in the numerator of \eqref{tsmall}.
However, 
\begin{equation}\label{tsmallar}
  \Var(X_n-X_n^N)\sim \Var(X_n^N)\sim 2\ga^2n\log n.
\end{equation}
\item 
Let $\sqrt{n}\ll N \ll \sqrt{n\log n}$, for example $N:=n\log\log n$.
Then the conclusions of \ref{tsmall12} still hold; moreover,
$\PP(X_n\neq X_n^N)\to0$.
  \end{thmenumerate}
\end{theorem}

The theorem thus shows that the large clades are rare, and do not contribute
to the asymptotic distribution; however, when they appear, the larges clades
give a large (actually negative) contribution to $X_n$, and as a result,
half the variance of $X_n$ comes from the large clades.
(When there is a large clade, there is less room for other clades, so $X_n$
tends to be smaller than usually. See also \eqref{ftv} and \eqref{f} below.)

For higher moments, the large clades play a similar, but even more extreme,
role. 
Note that (for $n\ge2$)
with probability $2/n$, 
the root of $\TT_n$ has one internal and
one external node, and then there is a clade consisting of all external nodes; 
this is obviously the unique maximal clade, and thus $X_n=1$. 
Since $\E X_n =\ga n+O(1)$ by
\refT{Tmean}, we thus have $X_n-\E X_n = -\ga n +O(1)$ with probability
$2/n$, and this single exceptional event gives a contribution
$\sim (-1)^k 2\ga^k n^{k-1}$ to $\E(X_n-\E X_n)^{k}$, which 
explains a fraction $(k-2)/k$ of the moment \eqref{Xmom}; 
in particular, this explains
why the moment is of order $n^{k-1}$.

We shall see later that, roughly speaking, the moment asymptotic
in \eqref{Xmom} is completely explained by extremely
large clades of size $\Theta(n)$,
which appear in the $O(1)$ first generations of the tree.

This will also lead to a version of \eqref{Xmom} for absolute central
moments:
\begin{theorem}
  \label{Tp}
For any fixed real $p>2$, as \ntoo,
  \begin{equation}\label{tp}
 \E\bigabs{X_n-\E X_n}^p \sim \frac{2p}{p-2}\ga^p n^{p-1}.
  \end{equation}
\end{theorem}

In \refS{Sbin}, we transfer the problem from random phylogenetic trees to 
\rbst, which we shall use in the proofs.
The theorems above are proved in Sections \ref{Smean}--\ref{Sga}.

\section{Binary trees}\label{Sbin}

We find it technically convenient to work with binary trees instead of full
binary trees (phylogenetic trees), so we use the bijection in \refS{S1} to
define $F(T)$ also for binary trees $T$. (We use the same notation $F$; this
should not cause any confusion.)
With this translation, our problem is thus to study $X_n:=F(\ctn)$, where $\ctn$
is the binary search tree with $n$ nodes.

The clades in a phylogenetic
tree correspond to the internal nodes that have at least one external child,
\ie, the nodes in the corresponding binary tree that have outdegree at most
1.
We call such nodes \emph{green}. 
For a binary tree $T$, 
the number $F(T)$ is thus the number of \emph{maximal green nodes}, 
\ie, the number
of green nodes that have no green ancestor. (This holds also for the 
phylogenetic tree $T$ with a single node, and thus for the empty binary
tree, with our definition $F(T)=0$ in this case.)

It follows that,
for any binary tree $T$, 
\begin{equation}\label{FT}
  F(T):=
  \begin{cases}
1 & \text{if $T$ has a green root},
\\
F(\tl)+F(\tr) & \text{otherwise}.	
  \end{cases}
\end{equation}

Define, for a binary tree $T$,
\begin{equation}\label{f}
  f(T):= F(T)-F(\tl)-F(\tr)=
  \begin{cases}
	1-F(\tr), & \tl=\emptyset,T\neq\emptyset, \\
	1-F(\tl), & \tr=\emptyset,T\neq\emptyset, \\
0, & \text{otherwise}.
  \end{cases}
\end{equation}
Then $F(T)$ is given by the recursion
\begin{equation}
  F(T)=F(\tl)+F(\tr)+f(T),
\end{equation}
and thus
\begin{equation}\label{ftv}
  F(T)=\sum_{v\in T}f(T_v),
\end{equation}
where $T_v$ is the subtree rooted at $v$, consisting of $v$ and all its
descendants. 
In another words, $F(T)$ is the additive functional defined by the toll
function $f(T)$. The advantage of this point of view is that we have
eliminated the maximality condition and now sum over all subtrees $T_v$, and
that we can use general results for this type of sums, see 
\citet{SJ296}.

We let $\cT$ denote the random binary search tree with a random number of
elements such that $\PP(|\cT|=n)=2/((n+1)(n+2))$, $n\ge1$. 
The random binary tree $\cT$
can be constructed by a continuous-time branching process: 
Let $(\ccT_t)_{t\ge0}$ be the growing tree that starts
with an isolated root at time $t=0$ and 
such that each existing node gets 
a left and a right child after random
waiting times that are independent and $\Exp(1)$;
we stop the process at a random time $\tau\sim\Exp(1)$,
independent of everything else, and can take $\cT=\ccT_\tau$, see 
\citet{Aldous-fringe} 
(where it is also proved that $\cT$ is the limit in
distribution of a random fringe tree in a binary search tree). 

\section{The mean}\label{Smean}

Recall that $\ctn$ is the random binary search tree with $n$ nodes. Define
$\nu_n:=\E F(\ctn)$ and $\mu_n:=\E f(\ctn)$, with $F$ and $f$ as in \refS{Sbin}.
(In particular, $\nu_0=\mu_0=0$, while $\nu_1=\mu_1=1$ since
$F(\cT_1)=f(\cT_1)=1$.) 
For $n\ge2$, $\ctnl$ is empty with probability $1/n$, and conditioned on
this event, $\ctnr$ has the same distribution as $\cT_{n-1}$. 
The same holds if we interchange $\sL$ and $\sR$.
Hence, taking
the expectation in \eqref{f},
\begin{equation}\label{munu}
  \mu_n = \tfrac{2}n\bigpar{1-\E F(\cT_{n-1})} 
= \tfrac{2}n\bigpar{1-\nu_{n-1}} ,
\qquad n\ge2.
\end{equation}
Furthermore, we see that \eqref{f} implies
\begin{equation}\label{max}
\PP\bigpar{f(\ctn)\neq0}\le \xfrac{2}n .
\end{equation}
Since obviously $0\le F(T)\le |T|$, we have by \eqref{f} also
$-|T|\le f(T)\le 1$ and thus 
\begin{equation}\label{maa}
|f(T)|\le |T|  
\end{equation}
for any binary tree $T$.
In particular, this and \eqref{max} yield
\begin{equation}
  \label{mua}
|\mu_n|\le\E|f(\ctn)|\le n\PP\bigpar{f(\ctn)\neq0}\le 2.
\end{equation}

It is now a simple consequence of general results that $\nu_n:=\E F(\ctn)$ is
asymptotically linear in $n$. Recall the random binary tree $\cT$ defined in
\refS{Sbin}. 
\begin{lemma}
  \label{LEF}
  \begin{equation}\label{nun}
\nu_n:=\E F(\ctn) = n\ga+O(1),	
  \end{equation}
where
\begin{equation}\label{lefa}
  \begin{split}
  \ga&:=\E f(\cT)
=\sumni \frac2{(n+1)(n+2)}\E f(\ctn)
=\sumni \frac2{(n+1)(n+2)}\mu_n
\\&\phantom:
=\sumni \frac4{n(n+1)(n+2)}(1-\nu_{n-1}).	
  \end{split}
\end{equation}
\end{lemma}
\begin{proof}
  An instance of \citet[Theorem 3.8]{SJ296}. 
More explicitly, see
\cite[Theorem 3.4]{SJ296},
\begin{equation}
  \label{jw}
\E F(\ctn) = (n+1) \sum_{k=1}^{n-1} \frac{2}{(k+1)(k+2)}\mu_k + \mu_n,
\end{equation}
which implies the result by \eqref{mua} and \eqref{munu}.
\end{proof}

In order to prove \refT{Tmean}, it remains to show that $\ga$ defined in
\eqref{lefa} equals $(1-e\qww)/4$ 
as asserted in \eqref{ga}. In other words, we need the following.

\begin{lemma}
  \label{Lga}
  \begin{equation}
	\E f(\cT)=\frac{1-e\qww}4.
  \end{equation}
\end{lemma}

We can prove \refL{Lga} by probabilistic methods, using the construction of
$\cT$ by a branching process in \refS{Sbin}.
However, this proof is
considerably longer than the proof of \refT{Tmean} by singularity analysis
of generating functions in \cite{DF10} and \cite{DFL}; we nevertheless find
the probabilistic proof interesting, and perhaps useful for future
generalizations, but since the methods in it are not needed for other
results in the present paper, we postpone our proof of \refL{Lga} to \refS{Sga}.

\section{Variance}

Let $\gamm_n:=\Var(f(\ctn))$ and $\gss_n:=\Var(F(\ctn))$.
Then $\gamm_0=\gamm_1=\gss_0=\gss_1=0$ and, for $n\ge2$, using \eqref{f},
\begin{equation}
  \label{noa}
\gamm_n = \E f(\ctn)^2 -\mu_n^2
=\frac{2}n\E\bigpar{F(\cT_{n-1})-1}^2-\mu_n^2
\le \frac{2}n n^2 =2n.
\end{equation}

Before proving the variance asymptotics in \eqref{Xvar},
we begin with a weaker estimate.

\begin{lemma}
  \label{Lgss0}
For $n\ge1$,
\begin{equation}
  \gss_n:=\Var F(\ctn) = O(n\log^2 n).
\end{equation}
\end{lemma}

\begin{proof}
  By \cite[Theorem 3.9]{SJ296}, where it suffices to sum to $n$ since we may
  replace $f(T)$ by $0$ for $|T|>n$ without changing $F(\ctn)$,
  \begin{equation}\label{gssO}
	\gss_n \le Cn\lrpar{\biggpar{\sumkn\frac{\gam_k}{k\qqc}}^2
	  +\sup_{k}\frac{\gamm_k}{k}+\sumkn\frac{\mu_k^2}{k^2}}
=O(n\log^2 n),
  \end{equation}
using \eqref{noa} and \eqref{mua}, provided $n\ge2$. The case $n=1$ is trivial.
\end{proof}

Write $f(T)=g(T)+h(T)$, where 
\begin{equation}\label{jg}
  g(T):=
  \begin{cases}
	1-\nu_{|T|-1}, & \tl=\emptyset,T\neq\emptyset
\text{ or }  \tr=\emptyset,T\neq\emptyset, \\
0, & \text{otherwise}.
  \end{cases}
\end{equation}
and thus, see \eqref{f},
\begin{equation}\label{jh}
  h(T):=
  \begin{cases}
	\nu_{|\tr|}-F(\tr), & \tl=\emptyset, \\
	\nu_{|\tl|}-F(\tl), & \tr=\emptyset, \\
0, & \text{otherwise}.
  \end{cases}
\end{equation}
Then $g(\cT_1)=1$,  $h(\cT_1)=0$, and, for $k\ge2$,
using \eqref{munu} and \eqref{mua},
\begin{align}
  \E g(\cT_k)&=\frac{2}{k}\bigpar{1-\nu_{k-1}}=\mu_k=O(1), \label{jeg}
\\  
\E h(\cT_k)&=\frac{2}{k}\E\bigpar{\nu_{k-1}-F(\cT_{k-1})}=0 \label{jeh},
\end{align}
{and, using \refL{Lgss0},}
\begin{equation}
  \Var h(\cT_k)=\frac{2}{k}\E\bigpar{\nu_{k-1}-F(\cT_{k-1})}^2
=\frac{2}k\gss_{k-1} = O(\log^2 k).
\label{jvh}  
\end{equation}

Let, for an arbitrary binary tree $T$,    
\begin{align}\label{GH}
G(T):=\sum_{v\in T} g(T_v)&&& \text{and}&& H(T):=\sum_{v\in T} h(T_v),  
\end{align}
so by \eqref{ftv}, 
\begin{equation}  \label{FGH}
F(T)=G(T)+H(T).
\end{equation}

\begin{lemma}
  \label{LG}
  For $n\ge1$, 
  \begin{align}
	\E G(\ctn)&=\nu_n, \label{lgG}\\
	\E H(\ctn)&=0, \label{lgH}\\
\Var H(\ctn)&=O(n). \label{lgVH} 
\end{align}
\end{lemma}
\begin{proof}
By \cite[Theorem 3.4]{SJ296}, \cf{} \eqref{jw}, and \eqref{jeh}, 
  \begin{equation}
\E H(\ctn) = (n+1) \sum_{k=1}^{n-1} \frac{2}{(k+1)(k+2)}\E h(\ctk) + 
\E h(\ctn) = 0,
  \end{equation}
which proves \eqref{lgH}. This implies \eqref{lgG}, 
since by \eqref{FGH},
\begin{equation}
\E G(\ctn)=\E F(\ctn)-\E H(\ctn) = \nu_n.  
\end{equation}

Similarly, 
by \cite[Theorem 3.9]{SJ296}, \cf{} \eqref{gssO}, and \eqref{jeh}--\eqref{jvh}, 
\begin{equation*}
  \Var H(\ctn) \le C n \lrpar{\biggpar{\sumki \frac{\log k}{k\qqc}}^2
+\sup_{k\ge1}\frac{\log^2 k}{k}+0}=O(n).
\qedhere
\end{equation*}
\end{proof}

We shall see that this means that 
$H(\ctn)$ is asymptotically negligible, and thus it suffices to
consider $G(\ctn)$.

Note that $g(T)$ depends only on the sizes $|\tl|$ and $|\tr|$. This
enables us to easily estimate the variance of $G(\ctn)$.

\begin{theorem}
  \label{TVG}
For all $n\ge1$,
\begin{equation}\label{tvg}
  \Var G(\ctn) = 4\ga^2 n\log n + O(n).
\end{equation}
\end{theorem}

\begin{proof}
  Write $g(T)=g(|T|,|\tl|,|\tr|)$. 
(We only care about $g(k,j,l)$ when $j+l=k-1$, but use three arguments for
  emphasis.) 
Thus $g(k,0,k-1)=g(k,k-1,0)=1-\nu_{k-1}$ and otherwise $g(k,j,k-j-1)=0$.
Let, as in \cite[Theorem 1.29]{SJ296}, $I_k$ be uniformly distributed
on \set{0,\dots,k-1} and
\begin{equation}\label{psi}
  \begin{split}
	\psi_k &:= \E\bigpar{\nu_{I_k}+\nu_{k-1-I_k}+g(k,I_k,k-1-I_k)-\nu_k}^2
\\&
=\frac{1}k\sum_{j=1}^{k-2}(\nu_{j}+\nu_{k-1-j}-\nu_k)^2
+\frac{2}k \bigpar{\nu_{k-1}+1-\nu_{k-1}-\nu_k}^2
\\&
=\frac{1}k\sum_{j=1}^{k-2}(\nu_{j}+\nu_{k-1-j}-\nu_k)^2
+\frac{2}k (\nu_{k}-1)^2
\\&
=O(1)+\frac{2}k\bigpar{\ga k+O(1)}^2 = 2\ga^2 k+O(1),
  \end{split}
\end{equation}
where we used that $\nu_j=\ga j+O(1)$ by \refT{Tmean}.
By \cite[Lemma 7.1]{SJ296}, then
\begin{equation}\label{varg}
  \begin{split}
\Var G(\ctn)& = (n+1)\sumkni\frac{2}{(k+1)(k+2)}\psi_k+\psi_n
\\&	
= (n+1)\sumkni\frac{4\ga^2 k+O(1)}{(k+1)(k+2)}+O(n)
= (n+1)\sumkni\frac{4\ga^2 }{k}+O(n)
\\&
=4\ga^2 n\log n+O(n).
  \end{split}
\raisetag{\baselineskip}
\end{equation}
\end{proof}

We can now prove \eqref{Xvar} in \refT{Tmom}. (Higher moments are treated in
\refS{Smom}.) 
\begin{theorem}
  \label{TVF}
For all $n\ge1$,
\begin{equation}\label{tvf}
  \Var F(\ctn) = 4\ga^2 n\log n + o(n\log n).
\end{equation}
\end{theorem}
This follows from \eqref{FGH},
\eqref{tvg} and \eqref{lgVH} by Minkowski's inequality
(the triangle inequality for $\sqrt{\Var{}}$).

\section{Asymptotic normality}\label{SCLT}

We prove the central limit theorem \refT{TCLT} by a martingale central limit
theorem for a suitable martingale that we construct in this section.

Consider the infinite binary tree $\Too$, where each node has two
children, and denote its root by $o$. 
We may regard any binary tree $T$ as a subtree of $\Too$ with the
same root $o$. (In the general sense that the node set $V(T)$ is a 
subset of $\Voo:=V(\Too)$, and that the left and right children are the same
as in $\Too$, when they exist.)
In particular we regard the \rbst{}
$\ctn$ as a subtree of $\Too$.

Order the nodes in $\Too$ in breadth-first order as $v(1)=o,v(2),\dots$, 
and let $V_j:=\set{v(1),\dots,v(j)}$ be the set
of the first $j$ nodes. Let $\cF_j$ be the $\gs$-field generated by the
sizes $|\cT_{n,v,\sL}|$ and $|\cT_{n,v,\sR}|$ of the two child 
subtrees of $\ctn$
at each node $v\in V_j$. Equivalently, we may regard $V_j$ as the internal
nodes in a full binary tree; let $\dV_j$ be the corresponding set of $j+1$
external nodes. Then $\cF_j$ is generated by the subtree sizes $|\cT_{n,v}|$
for all $v\in \dV_j$, together with the indicators $\ett{v\in \ctn}$,
$v\in V_j$, that describe $\ctn\cap V_j$. (We regard the subtree $\cT_{n,v}$
as defined for all $v\in \Voo$, with $\cT_{n,v}=\emptyset$ if $v\notin\ctn$.)
Then, conditioned on $\cF_j$, $\ctn$ consists of some given 
subtree of $V_j$
together with attached subtrees $\cT_{n,v}$ at all nodes $v\in\dV_j$; these
are independent \bst{s} of some given orders.

We allow here $j=0$; $V_0=\emptyset$ and $\cF_0$ is the trivial $\gs$-field.

\begin{remark}\label{Rsplit}
  As is well-known, see \eg{} \cite{Drmota}, another construction of the
  \rbst{} $\ctn$ ($n\ge1$)
is to let the random variable $I_n$ be uniformly
  distributed on \set{0,\dots,n-1}, and to let $\ctn$ be
  defined recursively such that, given $I_n$, 
$\ctnl$ and $\ctnr$ are independent binary
  search trees with $|\ctnl|=I_n$ and $|\ctnr|=n-1-I_n$. 
(When the tree is used to sort $n$ keys, $I_n$ tells how many of the keys
that are assigned to the left subtree.)
The pair $(I_n,n-1-I_n)$ thus tells how the tree is split at the root, and
there is a similar pair for each node. 
Then $\cF_j$ is generated by these pairs (\ie, splits) for the nodes
$v_1,\dots,v_j$. 
\end{remark}

Recall that $g(T)$ by \eqref{jg}
depends only on the sizes $|\tl|$ and $|\tr|$. Hence,
$\cF_j$ specifies the value of $g(\cT_{n,v})$ for every $v\in V_j$, and
it follows that
\begin{equation}\label{jepp}
  \E\bigpar{G(\ctn)\mid\cF_j}
= \E\Bigpar{\sum_{v\in \Voo} g(\cT_{n,v})\Bigm|\cF_j} 
= \sum_{v\in V_j} g(\cT_{n,v}) + \sum_{v\in \dV_j} \nu_{|\cT_{n,v}|}.
\end{equation}
Since the sequence of $\gs$-fields $(\cF_j)_0^\infty$ is increasing,
the sequence 
$M_{n,j}:=\E\bigpar{G(\ctn)\mid\cF_j}$, $j\ge0$, is a martingale (for any
fixed $n$). It follows from \eqref{jepp}
that the martingale differences are
\begin{equation}\label{gDM}
\Delta M_{n,j}
:=M_{n,j}-M_{n,j-1}
=g(\cT_{n,v(j)})+\nu_{|\ctnx{v(j)_\sL}|}  +\nu_{|\ctnx{v(j)_\sR}|}
-\nu_{|\ctnx{v(j)}|},    
\end{equation}
where $v(j)_\sL$ and $v(j)_\sR$ are the children of $v(j)$.
It follows easily that, with $\psi_k$ defined in \eqref{psi},
\begin{equation}
  \E\bigpar{|\Delta M_{n,j}|^2\mid\cF_{j-1}}
=
 \E\bigpar{|\Delta M_{n,j}|^2\mid |\cT_{n,v(j)}|}
=\psi_{|\ctnx{v(j)}|}.
\end{equation}
Consequently, the conditional square function is given by
\begin{equation}
  \begin{split}
W_n:=\sumji   \E\bigpar{|\Delta M_{n,j}|^2\mid\cF_{j-1}}
=\sum_{v\in\Voo}\psi_{|\ctnx{v}|}
=\sum_{v\in\ctn}\psi_{|\ctnx{v}|}.
  \end{split}
\end{equation}
(It suffices to sum over $v\in \ctn$, since  $\psi_0=0$.)
This is again a sum of the same type as \eqref{ftv} and \eqref{GH}, for the
random tree $\ctn$.
(Note that the toll function $\psi_{|T|}$ here depends only on the size
of $T$.)
In particular,
\cite[Theorem 3.4]{SJ296} applies 
(in this case we can also use
\cite{Devroye1},
\cite{Devroye2} or
\cite{FlajoletGM1997}); this yields
  \begin{equation}\label{ew}
\E W_n = (n+1) \sum_{k=1}^{n-1} \frac{2}{(k+1)(k+2)}\psi_k + 
\psi_n.
  \end{equation}
If $j$ is large enough, say $j\ge 2^n$, then $V(\ctn)\subseteq V_j$ and thus
$M_{n,j}=G(\ctn)$.  
In particular, $G(\ctn)=M_{n,\infty}$.
Thus, by a standard (and simple) martingale identity,
$\Var G(\ctn)=\Var M_{n,\infty}=\E W_n$; hence \eqref{ew} yields the first
equality in \eqref{varg}. (This is no coincidence; the proof just given of
\eqref{ew} is essentially the same as the proof of 
\cite[Lemma 7.1]{SJ296} that was used in \eqref{varg}, but stated in martingale
formulation.) 

We now split the sum $G(\ctn)$ into two parts, roughly corresponding to
small and large clades. We fix a cut-off $N=N(n)$; for definiteness and
simplicity we choose
$N=N(n):=\sqrt n$, but we note that the arguments below hold with a few
minor modifications for any $N\ge\sqrt n$ with $N=o(\sqrt{n\log n})$.
We then define, for binary trees $T$,
\begin{align}
  g'(T)&:= g(T)\ett{|T|\le N} \label{g'}\\
  g''(T)&:= g(T)\ett{|T|> N}=g(T)-g'(T). \label{g''}
\end{align}
In analogy with \eqref{ftv} and \eqref{GH}, we define further 
\begin{align}
G'(T):=\sum_{v\in T} g'(T_v)  
&&&
\text{and} && G''(T):=\sum_{v\in T} g''(T_v); 
\end{align}
thus 
$G(T)=G'(T)+G''(T)$. We shall see that, asymptotically,  both $G'(\ctn)$ and
$G''(T)$  contribute to the variance with equal amounts,
but nevertheless $G''(\ctn)$ is negligible (in probability).

We begin with the main term $G'(\ctn)$. 
\begin{lemma}\label{LG'}
  As \ntoo,
  \begin{align}
  \Var\bigpar{G'(\ctn)}
&= 2\ga^2 n\log n + O(n), \label{varg'}
\\	
\label{lg'}
	\frac{G'(\ctn)-\E G'(\ctn)}{\sqrt{2\ga^2 n\log n}}
&\dto N(0,1).	
  \end{align}
\end{lemma}

\begin{proof}
We define $\nu'_n:=\E G'(\ctn)$.
Note that $g'(T)$  depends only on the sizes 
$|\tl|$ and $|\tr|$.
Hence we can repeat the argument above and define a martingale
$M'_{n,j}:=\E\bigpar{G'(\ctn)\mid\cF_j}$, $j\ge0$, with
$G'(\ctn)=M'_{n,\infty}$ and martingale differences
\begin{equation}\label{blid}
  \Delta M'_{n,j}=\gff(\ctnx{v(j)}),
\end{equation}
where we define, \cf{} \eqref{gDM},
\begin{equation}\label{gff}
  \gff(T):=g'(T)+\nu'_{|\tl|}+\nu'_{|\tr|}-\nu'_{|T|}.
\end{equation}

By \cite[Theorem 3.4]{SJ296} again, \cf{} \eqref{jw} and \eqref{ew},
using $\E g(\ctk)=\mu_k=O(1)$ by \eqref{jeg},
\begin{equation}
  \begin{split}
\nu'_m&=(m+1)\sumkix{m-1}\pkk \E g'(\ctk)+\E g'(\cT_m)
\\&=(m+1)\sumkix{(m-1)\land N}\pkk \E g(\ctk)+O(1)
\\&=(m+1)\sumkix{N}\pkk \mu_k+O(1).
  \end{split}
\end{equation}
Hence, \eqref{gff} yields, after cancellations, 
\begin{equation}\label{gff2}
  \gff(T)=g'(T)+O(1)
=
\begin{cases}
  g(T)+O(1), & |T|\le N,
\\
O(1), & |T|>N.
\end{cases}
\end{equation}

Let 
\begin{equation}\label{psi'}
  \psi'_k:=\E |\gff(\ctk)|^2.
\end{equation}
Then, by \eqref{gff2}, \eqref{jg} and \eqref{nun},
\cf{} \eqref{psi},
\begin{equation}\label{paddington}
\psi'_k
=
\begin{cases}
\E\bigpar{g(\ctk)+O(1)}^2=2\ga^2 k+O(1), & k\le N,
\\
O(1), & k>N.
\end{cases}
\end{equation}

Furthermore, by \eqref{blid} and \eqref{psi'},
\begin{equation}
  \begin{split}
\E\bigpar{|\Delta M'_{n,j}|^2\mid\cF_{j-1}}	
=
\E\bigpar{|\gff(\ctnx{v(j)})|^2\mid |\ctnx{v(j)}|}	
=\psi'_{|\ctnx{v(j)}|}.
  \end{split}
\end{equation}
Hence, 
the conditional square function of $(M'_{n,j})_j$ is 
\begin{equation}\label{magno}
  \begin{split}
W'_n:=\sumji \E\bigpar{|\Delta M'_{n,j}|^2\mid\cF_{j-1}}
=\sum_{v\in\Voo}\psi'_{|\ctnx{v}|}
=\sum_{v\in\ctn}\psi'_{|\ctnx{v}|}.
  \end{split}
\end{equation}
Yet another application of \cite[Theorem 3.4]{SJ296} yields, using
\eqref{paddington}, 
\begin{equation}\label{elfbrink}
  \begin{split}
  \E W_n'&=(n+1)\sumkni\pkk\psi'_k+\psi'_n
  \\&
=(n+1)\sumkix{N}\frac{4\ga^2 k}{(k+1)(k+2)}+O(n) 	
\\&
=4\ga^2 n\log N+O(n) = 2\ga^2 n\log n+O(n).
  \end{split}
\end{equation}

Since $\Var G'(\ctn)= \Var\bigpar{M'_{n,\infty}}=\E W'_n$, 
\eqref{varg'} follows from \eqref{elfbrink}.

Moreover, the representation \eqref{magno} and
\cite[Theorem 3.9]{SJ296} (again summing only to $n$, as we may) yield, 
noting that the toll function $\psi'_{|T|}$
depends only on the size of $T$,
using \eqref{paddington},
\begin{equation}
  \label{anna}
\Var(W_n') 
\le C n \sumkn \frac{(\psi'_k)^2}{k^2}
\le C_1 n \sumkN 1 + 
C_2 n \sumkn \frac{1}{k^2}
=O(nN)=O(n^2).
\end{equation}
Hence, $\Var\bigpar{W'_n/(n\log n)}\to0$ as \ntoo, which together with
\eqref{elfbrink} implies
\begin{equation}\label{olof}
  \frac{W'_n}{n\log n} \pto 2\ga^2.
\end{equation}

Note also that $g(T)=O(|T|)$ by  \eqref{jg} and \eqref{nun}, and thus
\eqref{gff2} implies $\gff(T)=O(N)$ for all trees $T$. Thus \eqref{blid}
yields
\begin{equation}\label{ulla}
\sup_j\frac{\abs{\gD M_{n,j}}}{\sqrt{n\log n}}
=O\Bigparfrac{N}{\sqrt{n\log n}}
=o(1).
\end{equation}

We now apply the central limit theorem for martingale triangular arrays, in
the form in \cite[Corollary 1]{BrownE} (see also \cite[Theorem 3.1]{HH}),
which shows that \eqref{olof} and \eqref{ulla} together imply
\begin{equation}
  	\frac{G'(\ctn)-\E G'(\ctn)}{\sqrt{n\log n}}
=
	\frac{M_{n,\infty}-\E M_{n,\infty}}{\sqrt{n\log n}}
\dto N\bigpar{0,2\ga^2}.
\end{equation}
(Actually, \cite[Corollary 1]{BrownE} assumes instead of \eqref{ulla} only a
conditional Lindeberg condition, which is a trivial consequence of the
uniform bound
\eqref{ulla}.) 
\end{proof}

\begin{remark}
We used the breadth-first order above as just one convenient order. It is
perhaps more natural to consider instead of the sets $V_j$
arbitrary node sets $V$ of (finite)
subtrees of $\Too$ that include the root $o$. This would give us, instead of
$(M_{n,j})_j$, a
martingale indexed by binary trees. However, we have no use for this
exotic object here, and use instead the standard martingales 
above.
\end{remark}

\begin{lemma}
  \label{Llarge}
  \begin{align}
\E |G''(\ctn)| &=O\bigpar{\sqrt n}, \label{lle}\\
\Var(G''(\ctn)) &=2\ga^2 n\log n+O(n).\label{llvar}
  \end{align}
\end{lemma}

\begin{proof}
  By \eqref{g''}, \eqref{jg} and \eqref{jeg}, 
  \begin{equation}
\E|g''(\ctk)|
=|\E g(\ctk)|\cdot\ett{k>N}	
= O(1)\cdot\ett{k>N}
  \end{equation}
and thus, using the triangle inequality and \cite[Theorem 3.4]{SJ296},
\begin{equation*}
  \begin{split}
\E |G''(\ctn)| \le (n+1) \sum_{N}^{n-1}\pkk\E |g''(\ctk)|+\E|g''(\ctn)|
=O\Bigparfrac{n}{N}
,  \end{split}
\end{equation*}
yielding \eqref{lle}.

For the variance, we use either 
\cite[Theorem 1.29]{SJ296} 
as in the proof of \refT{TVF}, or the (essentially equivalent) martingale
argument in \eqref{blid}--\eqref{elfbrink} and conclude that, with
some $\psi''_k$ satisfying
\begin{equation}\label{padd}
\psi''_k
=
\begin{cases}
O(1), & k\le N,
\\
\E\bigpar{g(\ctk)+O(1)}^2=2\ga^2 k+O(1), & k>N,
\end{cases}
\end{equation}
we have
\begin{equation*}
  \begin{split}
\Var G''(\ctn)& = (n+1)\sumkni\frac{2}{(k+1)(k+2)}\psi''_k+\psi''_n
\\&	
= (n+1)\sum_{k=\floor N+1}^{n-1}\frac{4\ga^2 k}{k^2}+O(n)
\\&
=4\ga^2 n\log (n/N)+O(n)
=2\ga^2 n\log n+O(n). 
\qedhere
  \end{split}
\end{equation*}
\end{proof}

\begin{proof}[Proof of \refT{TCLT}]
  It follows from \eqref{lle} that
  \begin{equation}
	\frac{G''(\ctn)-\E G''(\ctn)}{\sqrt{2\ga^2 n\log n}}
\pto 0,	
  \end{equation}
which together with \eqref{lg'} yields 
  \begin{equation}\label{asnG}
	\frac{G(\ctn)-\E G(\ctn)}{\sqrt{2\ga^2 n\log n}}
\dto N(0,1).	
  \end{equation}

Similarly, \eqref{lgVH} implies
\begin{equation}
\frac{H(\ctn)-\E H(\ctn)}{\sqrt{2\ga^2 n\log n}}\pto0,  
\end{equation}
which together with \eqref{asnG} yields \eqref{tclt}, recalling
$X_n=F(\ctn)=G(\ctn)+H(\ctn)$ by \eqref{FGH}.
\end{proof}

\begin{proof}[Proof of \refT{Tsmall}]
(i).
Define, similarly to \eqref{g'}--\eqref{g''},
\begin{align}
  f'(T):= f(T)\ett{|T|\le N},  &&&
  f''(T):= f(T)\ett{|T|> N}, \label{f'}
\\
  h'(T):= h(T)\ett{|T|\le N}, &&&
  h''(T):= h(T)\ett{|T|> N} \label{h'},
\end{align}
  and corresponding sums 
$  F'(T):=\sum_{v\in T}f'(T_v)$ 
and similarly $F''(T)$, $H'(T)$, $H''(T)$.
The argument in \eqref{FT}--\eqref{ftv} is easily modified and shows that
\begin{equation}\label{FGH'}
  X_n^N=F'(\ctn)=G'(\ctn)+H'(\ctn).  
\end{equation}

The same proof as for \refL{LG} yields also
\begin{align}\label{mackmyra}
  \Var H'(\ctn)=O(n)&&& \text{and}&& \Var H''(\ctn)=O(n).  
\end{align}
Hence, \eqref{tsmall} follows from \refL{LG'} and \eqref{FGH'}.

Furthermore,
\begin{align}\label{FGH''}
  X_n-X_n^N = F''(\ctn)=G''(\ctn)+H''(\ctn).
\end{align}
By \eqref{FGH'} and \eqref{FGH''},  
\eqref{tsmallar} follows from \eqref{varg'} and \eqref{llvar}, 
using \eqref{mackmyra} and  Minkowski's inequality.
Similarly,
\begin{equation}
  \E|X_n-X_n^N| \le \E|G''(\ctn)|+\E|H''(\ctn)| = O(\sqrt n),
\end{equation}
using \eqref{lle}, \eqref{mackmyra} and \Holder's inequality, together with
$\E H''(\ctn)=0$, which is proved as \eqref{lgH}.

(ii).
The conclusions of (i) hold by the same proofs (with some minor
modifications in some estimates).

Moreover, let $Z_{n,k}$ be the number of clades of size $k+1$.
Then, for $n\ge2$,  the expected number is given by
\begin{equation}
  \E Z_{n,k} = 
  \begin{cases}
	\frac{4n}{k(k+1)(k+2)}, & k<n, \\
\frac{2}{n}, & k=n, \\
0, & k>n,
  \end{cases}
\end{equation}
see \cite[Theorem 1]{ChangF10}. (This can be seen as another example of
\cite[Theorem 3.4]{SJ296}.)
Consequently,
\begin{equation}
  \begin{split}
\PP(X_n\neq X_n^N)
&\le \PP\Bigpar{\sum_{k>N} Z_{n,k} \ge1}	
\\&
\le \E\sum_{k>N} Z_{n,k} 
=
\sum_{\floor N+1}^{n-1}\frac{4n}{k(k+1)(k+2)}+\frac2n
\\&
=O\Bigparfrac{n}{N^2}  + O\Bigparfrac1{n} =o(1),
  \end{split}
\end{equation}
which completes the proof.
\end{proof}

\section{Higher moments}\label{Smom}

We begin the proof of \refT{Tp} by proving a weaker estimate.
We let $\normp{X}:=(\E X^p)^{1/p}$ for any random variable $X$.
Recall that $\nu_n:=\E F(\ctn)$.
\begin{lemma}\label{Lxp}
  For any fixed real $p>2$, and all $n\ge1$,
  \begin{equation}\label{lxp}
\E \bigabs{F(\ctn)-\nu_n}^p
\le C(p) n^{p-1}.
  \end{equation}
Equivalently,
\begin{equation}\label{lxp2}
  \bignorm{F(\ctn)-\nu_n}_p 
= O(n\qpp).
\end{equation}
\end{lemma}

\begin{proof}
  Fix $p>2$ and let $m\ge1$ be chosen below. (The constants $C_i$ below may
  depend on $p$ but not on $m$.)
Let $V_j$ and $\cF_j$ be as in \refS{SCLT}, and write $V'_m:=V_{2^m-1}$,
$\cF'_m:=\cF_{2^m-1}$.
Thus $\dV'_m$ consists of the $2^m$ nodes in $\Too$ of depth $m$, and $V_m'$
consists of the $2^m-1$ nodes of smaller depth.
It follows from \eqref{ftv} that, for any binary tree $T$,
\begin{equation}\label{xa}
F(T)=\sum_{v\in V'_m} f(T_v) + \sum_{v\in \dV'_m} F(T_v).
\end{equation}
Furthermore, by \eqref{tmean}, 
\begin{equation}\label{xb}
  \begin{split}
  \sum_{v\in\dV'_m}\nu_{|T_v|} 
&=  \sum_{v\in\dV'_m}\bigpar{\ga {|T_v|} +O(1)}
=  \ga\sum_{v\in\dV'_m}|T_v| +O(2^m) 
\\&
=\ga|T| +O(2^m) 	
=\nu_{|T|} +O(2^m).
  \end{split}
\end{equation}
Hence, by combining \eqref{xa} and \eqref{xb},
\begin{equation}\label{isis}
F(T)-\nu_{|T|}
=\sum_{v\in V'_m} f(T_v) + \sum_{v\in \dV'_m} \bigpar{F(T_v)-\nu_{|T_v|}}
+O(2^m).
\end{equation}

We shall use this decomposition for the \bst{} $\ctn$.
Note first that by \eqref{max}--\eqref{maa},
\begin{equation}\label{osiris}
  \E |f(\ctn)|^p \le n^p\PP\bigpar{f(\ctn)\neq0} \le 2n^{p-1}.
\end{equation}
(This holds for any $p>0$ and generalises \eqref{mua} which is the case $p=1$.)
Hence, for any $v\in \Voo$,
\begin{equation}
  \E\bigpar{|f(\ctnv)|^p\bigm||\ctnv|}
\le 2|\ctnv|^{p-1}\le 2n^{p-1},
\end{equation}
and thus 
\begin{equation}\label{tor}
  \E|f(\ctnv)|^{p}
\le 2n^{p-1}.
\end{equation}
Let $Y:=\sum_{v\in V'_m}f(\ctnv)$ be the first sum in \eqref{isis} for $T=\ctn$.
By Minkowski's inequality and \eqref{tor},
\begin{equation}\label{oden}
  \norm{Y}_p
\le \sum_{v\in V_m'}\norm{f(\ctnv)}_p
\le 2^m 2^{1/p} n^{(p-1)/p}.
\end{equation}

Let
$Z:=\sum_{v\in\dV_m'}\bigpar{F(\ctnv)-\nu_{|\ctnv|}}$
be the second sum in \eqref{isis} for $T=\ctn$.
The $\gs$-field $\cF'_m$ specifies the sizes of the subtrees $\ctnv$
for $v\in\dV_m'$, and conditioned on $\cF'_m$, these subtrees are
independent and distributed as $\cT_{n(v)}$ of the given sizes $n(v)$. Hence, 
conditionally on $\cF_m'$, 
the terms in the sum $Z$ are independent and have means zero,
so we can apply Rosenthal's inequality 
\cite[Theorem 3.9.1]{Gut}, which yields
\begin{multline}\label{manne}
\E\bigpar{|Z|^p\mid\cF_m'}
\le \CCname{\CCmanne}\sumvdv\E \bigpar{\abs{F(\ctnv)-\nu_{|\ctnv|}}^p\mid\cF_m'}
\\
+ \CCx\Bigpar{\sumvdv\E \bigpar{\abs{F(\ctnv)-\nu_{|\ctnv|}}^2\mid\cF_m'}}^{p/2}
.  
\end{multline}
We note first that by \eqref{Xvar},
\begin{equation}
  \begin{split}
\E \bigpar{\abs{F(\ctnv)-\nu_{|\ctnv|}}^2\mid\cF_m'}
\le \CC |\ctnv|\log|\ctnv|	
\le \CCx |\ctnv|\log n,
  \end{split}
\end{equation}
and thus
\begin{equation}
  \begin{split}
\sumvdv\E \bigpar{\abs{F(\ctnv)-\nu_{|\ctnv|}}^2\mid\cF_m'}
\le \CCx \sumvdv  |\ctnv|\log n
\le \CCx n\log n.
  \end{split}
\end{equation}
Hence the second term on the \rhs{} in \eqref{manne} is 
$\le \CC (n\log n)^{p/2}$. Taking the expectation in \eqref{manne} we thus
obtain
\begin{equation}\label{frej}
  \E|Z|^p
\le \CCmanne \sumvdv\E {\abs{F(\ctnv)-\nu_{|\ctnv|}}^p}
+\CCname{\CCfrej}(n\log n)^{p/2}.
\end{equation}

Let $A_n:=\E|F(\ctn)-\nu_n|^p$. 
We can write \eqref{isis} for $T=\ctn$ as 
\begin{equation}\label{osis}
  F(\ctn)-\nu_n=Y+Z+O(2^m).
\end{equation}
Thus, by  Minkowski's inequality, \eqref{oden}
and \eqref{frej},
\begin{equation}\label{rom}
  \begin{split}
A_n& =\E\bigabs{Y+Z+O(2^m)}^p
\le 3^p\bigpar{\E |Y|^p+\E|Z|^p+O(2^m)}
\\&
\le \CC 2^{mp}n^{p-1}
+\CCname{\CCz}\E|Z|^p+\CC 2^m
\le 
\CCz\E|Z|^p
+\CCname{\CCrom} 2^{mp}n^{p-1}.
  \end{split}
\end{equation}
Furthermore, \eqref{frej} can be written
\begin{equation}\label{freja}
 \E|Z|^p
\le \CCmanne \sumvdv\E A_{|\ctnv|}
+\CCfrej(n\log n)^{p/2}.
\end{equation}

We prove the lemma by induction, and assume that 
$A_k\le Ck^{p-1}$ for all $k<n$. 
Since $|\ctnv|<n$ for every $v\in\dV_m'$, \eqref{freja} and the inductive
hypothesis yield
\begin{equation}\label{vidar}
 \E|Z|^p
\le \CCmanne C\sumvdv\E |\ctnv|^{p-1}
+\CCfrej(n\log n)^{p/2}.
\end{equation}
If $v$ is a child of the root, then $|\ctnv|$ is uniformly distributed on
\set{0,\dots,n-1}, so $|\ctnv|\eqd\floor{nU}\le nU$, where $U\sim U(0,1)$ is
uniformly distributed on $\oi$. By induction in $m$, it follows that for any
$v\in \dV_m'$,
\begin{equation}
  |\ctnv|\le n\prodim U_i,
\end{equation}
with $U_1,\dots,U_m$ independent and $U(0,1)$.
Consequently,
\begin{equation}\label{scott}
\E  |\ctnv|^{p-1} 
\le\E\Bigpar{ n^{p-1} \prodim U_i^{p-1}}
= n^{p-1} \prodim \E U_i^{p-1}
=  n^{p-1} (1/p)^{m},
\end{equation}
since $\E U_i^{p-1}=\intoi u^{p-1}\dd u=1/p$.
There are $2^m$ nodes in $\dV_m'$, and thus \eqref{vidar} yields 
\begin{equation}\label{balder}
 \E|Z|^p
\le \CCmanne C 2^m(1/p)^m n^{p-1}
+\CCfrej(n\log n)^{p/2},
\end{equation}
which together with \eqref{rom} yields, since $(n\log n)^{p/2}=O(n^{p-1})$
when $p>2$,
\begin{equation}\label{njord}
  \begin{split}
A_n 
&\le 
\CCz\CCmanne C(2/p)^m n^{p-1}
+\CCz\CCfrej(n\log n)^{p/2}
+{\CCrom} 2^{mp}n^{p-1}
\\
&\le 
\CCz\CCmanne C(2/p)^m n^{p-1}
+\CCname\CCnjord 2^{mp}n^{p-1}.
  \end{split}
\end{equation}

Now choose $m$ such that $(2/p)^m\CCz\CCmanne<1/2$ (which is possible
because $p>2$). Then choose $C:=2^{mp+1}\CCnjord$. With these choices,
\eqref{njord} yields
\begin{equation}
  A_n\le\tfrac12 C n^{p-1} +  \tfrac12 C n^{p-1}=  C n^{p-1}.
\end{equation}
In other words, we have proved the inductive step: 
$A_k\le C k^{p-1}$ for $k<n$ implies $A_n\le C n^{p-1}$.
Consequently, this is true for all $n\ge0$, \ie, \eqref{lxp} holds.
(The initial cases $n=0$ and $n=1$ are trivial, since $A_0=A_1=0$.)
\end{proof}

\begin{lemma}
  \label{L6}
For any fixed real $p>2$, as \ntoo,
\begin{align}
  \normp{F(\ctn)}& \sim \ga n ,\label{l6F}
\\
  \norm{f(\ctn)}_p& \sim 2\qp\ga n\qpp. \label{l6f}
\end{align}
\end{lemma}

\begin{proof}
By Minkowski's inequality, \eqref{lxp2} and \eqref{tmean},
\begin{equation}
  \bignorm{F(\ctn)}_p =\bigabs{\E F(\ctn)}+ O(n\qpp)
=\ga n+ O(n\qpp)\sim \ga n,
\end{equation}
which is \eqref{l6F}.

For $n\ge2$, it follows from \eqref{f} that
\begin{equation}\label{halma}
  \E|f(\ctn)|^p = \frac{2}n\E|1-F(\ctx{n-1})|^p
= \frac{2}n\norm{F(\ctx{n-1})-1}_p^p
\sim 2\ga^p n^{p-1},
\end{equation}
since \eqref{l6F} obviously implies also $\normp{F(\ctn)-1}\sim\ga n$.
\end{proof}

The idea in the proof of \refT{Tp} is to approximate
$\E|X_n-\E X_n|^p=\E\bigabs{\sum_v\bigpar{f(\ctnv)-\E f(\ctnv)}}^p$
by $\E\sum_v\bigabs{f(\ctnv)-\E f(\ctnv)}^p$, or simpler by
$\E\sum_v\bigabs{f(\ctnv)}^p=\sum_v\E\bigabs{f(\ctnv)}^p$. 
The heuristic reason for this is that the
moment $\E\bigabs{\sum_v\bigpar{f(\ctnv)-\E f(\ctnv)}}^p$ is dominated by
the event when there is one large term (corresponding to one large clade,
\cf{} the discussion before \refT{Tp}), and then 
\begin{equation}\label{app}
\Bigabs{\sum_v\bigpar{f(\ctnv)-\E f(\ctnv)}}^p
\approx \sum_v\bigabs{f(\ctnv)-\E f(\ctnv)}^p
  \approx \sum_v\abs{f(\ctnv)}^p.
\end{equation}

We shall justify this in several steps. We begin by finding the expectation
of the final sum in \eqref{app}, \cf{} the sought result  \eqref{tp}.
\begin{lemma}
  \label{Lbb} As \ntoo,
  \begin{equation}\label{lbb}
\E \sumvtn\abs{f(\ctnv)}^p\sim
 \frac{2p}{p-2}\ga^p n^{p-1}.	
  \end{equation}
\end{lemma}
\begin{proof}
 We apply again \cite[Theorem 3.4]{SJ296} and obtain
  \begin{equation}
	\begin{split}
\E \sum_{v\in\ctn}\abs{f(\ctnv)}^p= 
(n+1)\sumkni\pkk\E|f(\ctk)|^p+\E|f(\ctn)|^p.
	\end{split}
  \end{equation}
By \eqref{halma},
\begin{equation}
  \pkk\E|f(\ctk)|^p \sim \frac{2}{k^2}\cdot 2\ga^pk^{p-1}=4\ga^p k^{p-3}
\end{equation}
as \ktoo, and it follows that, as \ntoo, using $p>2$,
  \begin{equation*}
	\begin{split}
\E \sum_{v\in\ctn}\abs{f(\ctnv)}^p
&\sim
(n+1)\sumkni 4\ga^p k^{p-3}+2\ga^p n^{p-1}
\\&
\sim n\frac{4\ga^p}{p-2} n^{p-2}+2\ga^p n^{p-1}
= \frac{2p}{p-2}\ga^p n^{p-1}.
\qedhere
	\end{split}
  \end{equation*}
\end{proof}

Next we take again some $m\ge1$ and use the notation in the proof of
\refL{Lxp}. Since we now have proved \eqref{lxp}, the proof of \refL{Lxp}
shows that \eqref{balder} holds for every $n$, and thus, since $p>2$,
\begin{equation}
  \begin{split}
  \norm{Z}_p 
&\le \CC (2/p)^{m/p}n\qpp+O\bigpar{(n\log n)\qq}
\\&
= \CCx (2/p)^{m/p}n\qpp+o\bigpar{n\qpp}.	
  \end{split}
\end{equation}
Consequently, by \eqref{osis} and Minkowski's inequality,
\begin{equation}\label{this}
\bigabs{ \norm{F(\ctn)-\nu_n}_p - \norm{Y}_p}
\le \norm{Z}_p+O(2^m)
= \CCx (2/p)^{m/p}n\qpp+o\bigpar{n\qpp}.
\end{equation}
In particular, \eqref{this} and \eqref{lxp2} imply 
$ \norm{Y}_p= O(n\qpp)$.
By the mean value theorem,
\begin{equation}
  \label{mean}
|x^p-y^p|\le p|x-y|\max\set{x^{p-1},y^{p-1}}
\end{equation}
for any $x,y\ge0$; hence \eqref{this} implies, using also \eqref{lxp2} again,
\begin{equation}\label{erika}
  \begin{split}
\E| F(\ctn)-\nu_n|^p - \E |Y|^p
=O\bigpar{(2/p)^{m/p}n^{p-1}}+o\bigpar{n^{p-1}}.
  \end{split}
\end{equation}

Let $\gd>0$ be a small positive number to be chosen later, and 
let $J_v$ be the indicator of the event that $v$ is green and
$|\ctnv|\ge\gd n$. 
(The idea is that the significant contributions only come from nodes $v$
with $J_v=1$.)
\begin{lemma}\label{Ljm}
For each fixed $m\ge1$ and $\gd>0$, and all $n\ge1$,
\begin{align}
  \PP\Bigpar{\sumvv J_v \ge 1}& 
\le 2^{m+1}\gd\qw n\qw
= O\bigpar{n\qw}
,\label{ljm1}
\\
\PP\Bigpar{\sumvv J_v \ge 2}& 
\le 2^{2m+1}\gd\qww n\qww
= O\bigpar{n\qww}
.\label{ljm2}
\end{align}
\end{lemma}

\begin{proof}
We use again the $\gs$-fields $\cF_j$ from \refS{SCLT}. Since $\cF_{j-1}$
specifies $|\ctnx{v_j}|$, but not how this subtree is split at $v_j$, we
have 
\begin{equation}\label{jm}
  \PP(J_{v_j}=1\mid \cF_{j-1}) \le \frac{2}{|\ctnx{v_j}|}
  \ett{|\ctnx{v_j}|\ge\gd n} \le \frac{2}{\gd n},
\end{equation}
and thus, by taking the expectation,
$\PP(J_{v_j}=1)\le 2/(\gd n)$.  Since there are $<2^m$ nodes in
$V_m'$, \eqref{ljm1} follows.

Furthermore, for any two nodes $v_i$ and $v_j$ with $i<j$, 
$J_{v_i}$ is determined by $\cF_{j-1}$, and \eqref{jm} thus gives also
\begin{equation}
  \PP(J_{v_i}J_{v_j}=1\mid \cF_{j-1}) 
=  \E(J_{v_i}J_{v_j}\mid \cF_{j-1}) 
=J_{v_i} \PP(J_{v_j}=1\mid \cF_{j-1}) 
\le \frac{2}{\gd n}J_{v_i}.
\end{equation}
 Thus, by taking the expectation and using \eqref{jm} again, 
$  \PP(J_{v_i}J_{v_j}=1) \le 4/(\gd n)^2$. Summing over the less than
$\binom{2^m}2<2^{2m-1}$ pairs $(v_i,v_j)$ with $v_i,v_j\in V_m'$ yields 
\eqref{ljm2}.
\end{proof}

\begin{proof}[Proof of \refT{Tp}]
We show this in several steps.

\stepx
Define
\begin{equation}\label{y1}
  Y_1:=\sum_{v\in V_m'} J_vf(\ctnv).
\end{equation}
Since $f(\ctnv)=0$ unless $v$ is green, we have
\begin{equation}\label{regn}
  Y-Y_1=\sum_{v\in V_m'} (1-J_v)f(\ctnv)
=\sum_{v\in V_m'} f(\ctnv)\ett{|\ctnv|<\gd n}.
\end{equation}
For each $v$, it follows from \eqref{osiris}
by conditioning on $|\ctnv|$ that
\begin{equation}\label{mowgli}
  \E\bigabs{f(\ctnv)\ett{|\ctnv|<\gd n}}^p\le 2{(\gd n)^{p-1}}.
\end{equation}
Hence, \eqref{regn} and Minkowski's inequality yield
\begin{equation}\label{sno}
  \begin{split}
\bigabs{\norm{Y}_p-\norm{Y_1}_p}
&\le
\norm{Y-Y_1}_p
\le\sum_{v\in V_m'}\norm{ f(\ctnv)\ett{|\ctnv|<\gd n}}_p
\\&
\le {2^{m+1/p} (\gd n)\qpp}.	
  \end{split}
\end{equation}
Thus $\normp{Y_1}=O(n\qpp)+O(2^m\gd\qpp n\qpp)$, and \eqref{mean} yields
\begin{equation}\label{magnus}
  \E |Y|^p -\E |Y_1|^p = O\bigpar{(2^m\gd\qpp+2^{mp}\gd^{p-1})n^{p-1}}.
\end{equation}

\stepx
Similarly, using \eqref{mowgli} again,
\begin{equation}\label{jesper}
  \begin{split}
\E\Bigpar{\sum_{v\in V_m'} |f(\ctnv)|^p
-\sum_{v\in V_m'} J_v|f(\ctnv)|^p}
&
=\sum_{v\in V_m'}\E \bigpar{|f(\ctnv)|^p\ett{|\ctnv|<\gd n}}
\\&
\le{ 2^{m+1} (\gd n)^{p-1}}.
  \end{split}
\raisetag{\baselineskip}
\end{equation}

\stepx
  By \eqref{y1}, 
$|Y_1|^p -  \sumvv|J_vf(\ctnv)|^p =0$ unless $\sumvv J_v\ge2$,
and in the latter case we have by \eqref{maa}  the trivial bounds
$|Y_1|^p\le (2^m n)^p$ 
and $\sumvv|J_vf(\ctnv)|^p \le 2^m n^p$, and thus 
$\bigabs{|Y_1|^p -  \sumvv|J_vf(\ctnv)|^p}\le 2^{mp}n^p$.
Consequently, by \eqref{ljm2}, 
\begin{equation} 
\E \Bigabs{|Y_1|^p - \sumvv|J_vf(\ctnv)|^p}
\le 2^{mp}n^p \PP\Bigpar{\sumvv  J_v\ge2}
=O(n^{p-2}).
\end{equation}
Thus,
for fixed $m\ge1$ and $\gd>0$,
\begin{equation}\label{emma}
  \E|Y_1|^p - \sumvv\E |J_vf(\ctnv)|^p 
= O\bigpar{n^{p-2}}
= o\bigpar{n^{p-1}}.
\end{equation}

\stepx
Define $\FF(T):=\sum_{v\in T}|f(T_v)|^p$. Then, in analogy with \eqref{xa}, 
\begin{equation}\label{xap}
\FF(T)=\sum_{v\in V'_m} |f(T_v)|^p + \sum_{v\in \dV'_m} \FF(T_v).
\end{equation}
Note that \refL{Lbb} implies $\E\FF(\ctn)=O(n^{p-1})$.
Hence, by first conditioning on $\cF'_m$, and using \eqref{scott},
\begin{equation}
  \begin{split}
\E \sum_{v\in \dV'_m} \FF(\ctnv) \le \CC \E \sumvdv|\ctnv|^{p-1}
=\CCx (2/p)^m n^{p-1}.	
  \end{split}
\end{equation}
Taking $T=\ctn$ in \eqref{xap} and taking the expectation, we thus find
\begin{equation}\label{sofie}
  \begin{split}
\E \sum_{v\in \ctn}|f(\ctnv)|^p	
- \E \sumvv|f(\ctnv)|^p	
= O\bigpar{(2/p)^m n^{p-1}}.
  \end{split}
\end{equation}

\stepx
Finally,
combining \eqref{erika}, \eqref{magnus}, \eqref{emma}, \eqref{jesper},
\eqref{sofie} and \eqref{lbb}, we obtain
\begin{equation}
  \begin{split}
\E| F(\ctn)-\nu_n|^p 
&=
 \frac{2p}{p-2}\ga^p n^{p-1}
+ O\bigpar{(2/p)^{m/p}n^{p-1}}
+O\bigpar{2^m\gd\qpp n^{p-1}}
\\&
\qquad
+O\bigpar{2^{mp}\gd^{p-1}n^{p-1}}
+o(n^{p-1}).
  \end{split}
\raisetag{\baselineskip}
\end{equation}
For any $\eps>0$, we can make each of the error terms on the \rhs{} less than
$\eps n^{p-1}$ by first choosing $m$ large and then $\gd$ small, and finally
$n$ large.
Consequently,
$\E| F(\ctn)-\nu_n|^p = \frac{2p}{p-2}\ga^p n^{p-1}+o(n^{p-1})$.
\end{proof}

\begin{proof}[Proof of \eqref{Xmom}]
 Now $p=k$ is an integer.
If $k$ is even, then \eqref{Xmom} is the same as \eqref{tp}, so we may
assume that $p=k\ge3$ is odd.

In this case, \eqref{mean} holds for all real $x,y$. 
 Thus for any random
variables $X$ and $Y$, using also \Holder's inequality,
  \begin{equation}
	\begin{split}
\E|X^p-Y^p|
&\le p\E\bigpar{|X-Y|\,|X|^{p-1}+|X-Y|\,|Y|^{p-1}}
\\&
\le p \normp{X-Y}\bigpar{\normp{X}^{p-1}+\normp{Y}^{p-1}}.
	\end{split}
  \end{equation}
It is now easy to modify the proof of \refT{Tp} and obtain
\begin{equation}\label{scar}
\E\bigpar{F(\ctn)-\nu_n}^p 
=
  \E \sumvtn{f(\ctnv)}^p + o\bigpar{n^{p-1}}.
\end{equation}
Furthermore, it follows from \eqref{f} that $f(T)\le0$ unless $|T|=1$.
Hence,
\begin{equation}\label{lett}
 \sumvtn{f(\ctnv)}^p =
-\sumvtn\abs{f(\ctnv)}^p + O(n).
\end{equation}

The estimate \eqref{Xmom} now follows from \eqref{scar}, \eqref{lett} and
\eqref{lbb}. 
\end{proof}

\section{Proof of \refL{Lga}}\label{Sga}

Define a \emph{chain} of length $k$ in a (binary) tree $T$
to be a sequence of $k$ nodes $v_1\dotsm v_k$ such that $v_{i+1}$ is a (strict)
descendant of $v_i$ for each $i=1,\dots,k-1$. In other words, $v_1,\dots,v_k$ are some
nodes (in order) on some path from the root. We say that the chain
$v_1\dotsm v_k$ is \emph{green} if
all nodes $v_1,\dots,v_k$ are green. (The nodes between the $v_i$'s may have
any colour.)

For a binary tree $T$ and $k\ge1$, 
let $F_k(T)$ be the number of green chains $v_1\dotsm v_k$ in $T$, 
and let $f_k(T)$
be the number of such chains where $v_1$ is the root.
Obviously, \cf{} \eqref{ftv},
\begin{equation}\label{fktv}
  F_k(T)=\sum_{v\in T} f_k(T_v).
\end{equation}
These functionals are useful to us because of the following simple
relations, that are cases of  inclusion-exclusion. 
\begin{lemma}\label{L2}
  For any binary tree $T$,
  \begin{align}
	f(T)&=\sumki (-1)^{k-1} f_k(T),\label{fksum}
\\
	F(T)&=\sumki (-1)^{k-1} F_k(T). \label{Fksum}
  \end{align}
\end{lemma}
\begin{proof}
Let $v$ be a node in $T$ and consider the contribution to the sum in 
\eqref{Fksum}
of all chains with final node $v_k=v$. This is clearly 0 if 1 if $v$ is not
green, and it is 1 if $v$ is a maximal green node; furthermore, if $v$ is
green but has $j\ge1$ green ancestors, then the contribtion is easily seen
to be $\sum_{i=0}^j\binom ji(-1)^i=(1-1)^j=0$. Hence the \rhs{} of \eqref{Fksum} is the number of
maximal green nodes, \ie, $F(T)$.

For \eqref{fksum} we can argue similarly: Both sides are 0 unless the root
$o$ is green. If it is, the chain $o$ gives contribution 1, and by 
inclusion-exclusion, the chains with a given final node $v\neq o$ yield
together a ycontribution $-1$ 
if $v$ is green and there are no green nodes between $v$ and $o$, and 0
otherwise. Hence the sum equals $f(T)$ by \eqref{f}.
(Alternatively, \eqref{fksum} follows by induction from \eqref{Fksum},
\eqref{ftv} and \eqref{fktv}.)
\end{proof}

\begin{lemma}\label{L3} 
For every $k\ge1$,
  \begin{equation}
  \E f_k(\cT)=\frac{k(k+3)}{(k+1)(k+2)}\cdot\frac{2^{k-1}}{k!}
= \frac{2^{k-1}}{k!} - \frac{2^{k}}{(k+2)!}.	
  \end{equation}
\end{lemma}

\begin{proof}
We use the construction of $\cT=\cttau$ in \refS{Sbin}, which we formulate as
follows. Consider again the infinite binary tree $\Too$,
and grow $\ctt$ as a subtree of $\Too$, \cf{} \refS{SCLT}. 
To do this, we
equip each node $v$ in $\Too$ with two clocks $\Cl(v)$ and
$\Cr(v)$. These are started when $v$ is added to the growing tree $\ctt$,
and each chimes after a random time with an exponential distribution with
mean 1; when the clock chimes we add a left or right child, respectively, to
$v$. 
There is also a \emph{doomsday clock} $C_0$, started at 0 and with the same
$\Exp(1)$ distribution; when it chimes (at time $\tau$),
the process is stopped and the tree $\cttau$ is output. All clocks are
independent of each other.

  Fix a chain $\vvk$ in the infinite tree $\Too$, with $v_1=o$, the root.
Let $\ell_i\ge0$ be the number of nodes between $v_i$ and $v_{i+1}$.
We compute the probability that $\vvk$ is a green chain in $\cT=\cttau$ 
by following the construction of $\ctt$ as time progresses,
checking in several steps
whether still $\vvk$ is a candidate for a green chain, and
computing the probability of this.
(We use throughout the proof the Markov property and the memoryless property
of the exponential distribution.)
We assume for notational convenience that the path from $v_1$ to $v_k$
always uses the left child of each node.
(By symmetry, this does not affect the result.) 

\step1
If $k>1$, we first need that $v_1=o$ has a left child but no right child (in
order to be green); in particular, of the three clocks
$\Cl(v_1)$, $\Cr(v_1)$, $C_0$ that run from the beginning, $\Cl(v_1)$ has to
chime first. This has probability $1/3$.

\step2
Given that Step 1 succeeds, 
$v_1$ gets a left child $w_1$. If $\ell_1>0$, we need a left
child of $w_1$, and still no right child at $v_1$. (But we do not care
whether we get a right child at $w_1$ or not.) Hence we need that 
$\Cl(w_1)$ chimes first among the
three clocks $\Cl(w_1)$, $\Cr(v_1)$, $C_0$ (ignoring all other clocks). 
This has probability $1/3$. 

This is repeated for $\ell_1$ nodes; thus, the total probability that steps 1
and 2 succeed is $3^{-(\ell_1+1)}$.

\step3
This takes us to $v_2$. If $k>2$, we need a left child but no right child at
$v_2$, and still no right child at $v_1$. Hence, the next chime from the
four clocks  $\Cl(v_2)$, $\Cr(v_2)$, $\Cr(v_1)$, $C_0$ has to come from
$\Cl(v_2)$. This has probability $1/4$. 

\step4
Similarly for each of the $\ell_2$ nodes between $v_2$ and $v_3$; again the
probability of success at each of these nodes is $1/4$. Hence the
probability that Steps 3 and 4 succeed is $4^{-(\ell_2+1)}$.

\step5
Steps 3 and 4 are repeated for $v_i$ for each $i<k$, yielding a probability
$(i+2)^{-(\ell_i+1)}$ of success for each $i$.

\step6
Finally, we have obtained $v_k$, and wait for the doomsday clock. Until it
chimes, we must not get any right child at $v_1,\dots,v_{k-1}$, and we must
get at most one child at $v_k$.
Hence, among the $k+2$ clocks $\Cr(v_1),\dots,\Cr(v_k)$, $\Cl(v_k)$, $C_0$,
the next chime must be either from $C_0$ (probability $1/(k+2)$), or from
$\Cl(v_k)$ or $\Cr(v_k)$, followed by $C_0$ (probability
$\frac{2}{k+2}\cdot\frac{1}{k+1})$. The probability of success in this step
is thus
\begin{equation}
\frac{1}{k+2}+  \frac{2}{k+2}\cdot\frac{1}{k+1}
=
 \frac{k+3}{(k+1)(k+2)}.
\end{equation}

Combining the six steps above, we see that the probability that $\vvk$ is a
green chain in $\cttau$ is
\begin{equation}
\frac{k+3}{(k+1)(k+2)}\prod_{i=1}^{k-1} \Bigparfrac{1}{i+2}^{\ell_i+1}.  
\end{equation}

Given $\ell_1,\dots,\ell_{k-1}$, there are $\prod_{i=1}^{k-1} 2^{\ell_i+1}$
choices of the chain $\vvk$, all with the same probability, so summing over
all $\ell_1,\dots,\ell_{k-1}\ge0$, we obtain
\begin{equation*}
  \begin{split}
\E f_k(\cT)
&=
\frac{k+3}{(k+1)(k+2)}\prod_{i=1}^{k-1} 
 \sum_{\ell_i=0}^\infty\Bigparfrac{2}{i+2}^{\ell_i+1}
=
\frac{k+3}{(k+1)(k+2)}\prod_{i=1}^{k-1} 
\frac{2}{i}
\\&
=
\frac{k+3}{(k+1)(k+2)}\cdot\frac{2^{k-1}}{(k-1)!}
=
\frac{k(k+3)}{(k+1)(k+2)}\cdot\frac{2^{k-1}}{k!}.
  \end{split}
\qedhere
\end{equation*}
\end{proof}

\begin{proof}[Proof of \refL{Lga}]
  By Lemmas \ref{L2} and \ref{L3}, and a simple calculation,
  \begin{equation*}
	\begin{split}
\E f(\cT) = \sumki (-1)^{k-1} \E f_k(\cT)	  
= \sumki\lrpar{\frac{(-2)^{k-1}}{k!} + \frac{(-2)^{k}}{(k+2)!}}
=\frac{1-e^{-2}}{4},
	\end{split}
\end{equation*}
noting that we may take the expectation inside the sum since
it also follows from \refL{L3} that
$\sumki  \E |f_k(\cT)| = \sumki  \E f_k(\cT)<\infty$.
\end{proof}

Recall that this, together with \refL{LEF}, completes our probabilistic
proof of \refT{Tmean}.

\begin{remark}
If we in the proof above change the doomsday clock and let it have an
arbitrary  rate  $\gl>0$, and denote the resulting random binary tree by
$\ctgl$, then the same argument yields
\begin{equation}
  \begin{split}
\E f_k(\ctgl)
&=
\frac{k+\gl+2}{(k+\gl)(k+\gl+1)}\prod_{i=1}^{k-1} 
 \sum_{\ell_i=0}^\infty\Bigparfrac{2}{i+\gl+1}^{\ell_i+1}
\\&
=
\frac{k+\gl+2}{(k+\gl)(k+\gl+1)}\prod_{i=1}^{k-1} 
\frac{2}{i+\gl-1}
\\&
=
\frac{(k+\gl-1)(k+\gl+2)}{(k+\gl)(k+\gl+1)}\frac{2^{k-1}}{\gl\rise k}
\\&
=\frac{2^{k-1}}{\gl\rise k}-\frac{2^{k}}{\gl\rise {k+2}}.
  \end{split}
\end{equation}
Thus by \refL{L2}, 
letting $\Fii$ denote the confluent hypergeometric function, see
\eg{} \cite[\S\S13.1--13.2 and 16.1--16.2]{NIST},
  \begin{equation}\label{swf}
	\begin{split}
\E f(\ctgl) 
&= \sumki (-1)^{k-1} \E f_k(\ctgl)
= \sumki\lrpar{\frac{(-2)^{k-1}}{\gl\rise k}+\frac{(-2)^{k}}{\gl\rise{k+2}}}
\\&
=-\frac12\bigpar{\Fii(1;\gl;-2)-1}
+\frac14\Bigpar{\Fii(1;\gl;-2)
-\Bigpar{1-\frac{2}{\gl}+\frac{2\cdot2}{\gl(\gl+1)}}}
\\&
=\frac{1}4+\frac{\gl-1}{2\gl(\gl+1)}-\frac14\Fii(1;\gl;-2).
	\end{split}
\raisetag\baselineskip
\end{equation}
Furthermore, if $\gl>1$ we can compute $\E F(\ctgl)$ by the same method; the
only difference is that we also allow a path of length $\ell_0\ge0$ from the
root 
to $v_1$, which gives an additional factor $(1+\gl)^{-\ell_0}$ for each
$\vvk$, leading to
\begin{equation}
  \E F_k(\ctgl)=\sum_{\ell_0=0}^\infty \parfrac{2}{\gl+1}^{\ell_0} \E f_k(\ctgl)
=\frac{\gl+1}{\gl-1}\E f_k(\ctgl),
\end{equation}
and hence, using both parts of \refL{L2},
\begin{equation}\label{swF}
  \E F(\ctgl)
=\sumki(-1)^{k-1}\E F_k(\ctgl)
=\frac{\gl+1}{\gl-1}\E f(\ctgl)
.
\end{equation}

Moreover, a simple argument shows that, for any $n\ge1$,
\begin{equation}
  \PP(|\ctgl|=n)=\prod_{i=2}^n\frac{i}{i+\gl}\cdot\frac{\gl}{n+1+\gl}
=\frac{\gl n!}{(2+\gl)\rise n},
\end{equation}
and conditioned on $|\ctgl|=n$, $\ctgl$ has the same distribution as $\ctn$,
\ie, 
$(\ctgl\mid|\ctgl|=n)\eqd\ctn$. Hence,
\begin{equation}
\E F(\ctgl) = \sumni \frac{\gl n!}{(2+\gl)\rise n} \nu_n,
\end{equation}
which can be interpreted as an unusual type of generating function for the
sequence $(\nu_n)$; note that \eqref{swF} and \eqref{swf} yield an explicit
expression for it.
\end{remark}

\newcommand\AAP{\emph{Adv. Appl. Probab.} }
\newcommand\JAP{\emph{J. Appl. Probab.} }
\newcommand\JAMS{\emph{J. \AMS} }
\newcommand\MAMS{\emph{Memoirs \AMS} }
\newcommand\PAMS{\emph{Proc. \AMS} }
\newcommand\TAMS{\emph{Trans. \AMS} }
\newcommand\AnnMS{\emph{Ann. Math. Statist.} }
\newcommand\AnnPr{\emph{Ann. Probab.} }
\newcommand\CPC{\emph{Combin. Probab. Comput.} }
\newcommand\JMAA{\emph{J. Math. Anal. Appl.} }
\newcommand\RSA{\emph{Random Struct. Alg.} }
\newcommand\ZW{\emph{Z. Wahrsch. Verw. Gebiete} }
\newcommand\DMTCS{\jour{Discr. Math. Theor. Comput. Sci.} }

\newcommand\AMS{Amer. Math. Soc.}
\newcommand\Springer{Springer-Verlag}
\newcommand\Wiley{Wiley}

\newcommand\vol{\textbf}
\newcommand\jour{\emph}
\newcommand\book{\emph}
\newcommand\inbook{\emph}
\def\no#1#2,{\unskip#2, no. #1,} 
\newcommand\toappear{\unskip, to appear}

\newcommand\arxiv[1]{\url{arXiv:#1}}
\newcommand\arXiv{\arxiv}

\end{document}